\documentclass{amsart}[12pt]
\usepackage{amssymb,amsmath,latexsym,amscd,amsfonts}
\usepackage{graphics,mathtools}
\usepackage{epsfig}
\usepackage[usenames,  dvipsnames]{color}

\newtheorem{thm}{Theorem}
\newtheorem{lem}[thm]{Lemma}

\newtheorem{dfn}[thm]{Definition}

\def\hpic #1 #2 {\mbox{$\begin{array}[c]{l} \epsfig{file=#1,height=#2} \end{array}$}}
\def\vpic #1 #2 {\mbox{$\begin{array}[c]{l} \epsfig{file=#1,width=#2} \end{array}$}}
\def\ignore #1 {}

\def\R{\mbox{{$\mathbb{R}$}}}
\def\Z{\mbox{{$\mathbb{Z}$}}}
\def\N{\mbox{{$\mathbb{N}$}}}

\def\eps{\epsilon}

\def \aaron #1 {\marginpar{#1 -AA}}

\title{Dull Cutoff for Circulants}
\date{\today}
\begin{document}
\begin{abstract}
Families of symmetric simple random walks on Cayley graphs of Abelian groups with a bound on the number of 
generators are shown to never have sharp cutoff in the sense of [1], [3], or [5].  Here convergence to the stationary
distribution is measured in the total variation norm.  This is a situation of bounded degree and no expansion;
sharp cutoff (or the cutoff phenomenon) has been shown to occur in families such as random walks on a hypercube 
[1] in which the degree is unbounded as well as on a random regular graph where the degree is fixed, but there 
is expansion [4].  Our examples agree with Peres' conjecture in [3] relating sharp cutoff, spectral gap, and mixing time.  
\end{abstract}

\author[Abrams]{Aaron Abrams}
\address[Aaron Abrams]{Washington and Lee University}
\email{abramsa@wlu.edu}

\author[Babson]{Eric Babson}
\address[Eric Babson]{University of California, Davis}
\email{babson@math.ucdavis.edu}

\author[Landau]{Henry Landau}
\address[Henry Landau]{AT\&T Research}
\email{henry.j.landau@gmail.com}

\author[Landau]{Zeph Landau}
\address[Zeph Landau]{University of California, Berkeley}
\email{zeph.landau@gmail.com}

\author[Pommersheim]{Jamie Pommersheim}
\address[Jamie Pommersheim]{Reed College}
\email{jamie@reed.edu}

\maketitle
\section{Introduction}

In this work we consider a finite Abelian group $G$ equipped with a generating set $\{a_i \}_{1\leq i \leq r}$ of 
size $r$.  We are interested in analyzing an aspect of the random walk given by applying one of the elements 
$\{0,  \pm a_i \}$ with equal probability; we'll call this a {\it type $r$} walk.  Here is our main result:


\begin{thm}
No family of walks all of the same type has sharp cutoff. 
\end{thm}
This result is known in the case $r=1$, i.e., for cyclic groups with only one step size.
See Section 8 of [5] for more on the convergence rate of type $r$ walks.  

Sharp cutoff is defined as follows.
If $A$ is an irreducible symmetric Markov matrix with unique stationary distribution ${\bf v}_0$ (so that 
$A{\bf v}_0={\bf v}_0$ and $|{\bf v}_0|_1=1$) and ${\bf x_0}=(1,0,\ldots,0)$, we write \[d_A(t)=|A^t{\bf x_0}-{\bf v}_0|_1,\]
for the distance to the stationary distribution at time $t$ and 
\[t_A(d)=\max\{t|d_A(t)\geq d\},\] for the time it takes to get within distance $d$ of the stationary distribution.  

\begin{dfn}\label{def}
A family $\{A_i\}$ of irreducible symmetric Markov matrices has sharp cutoff if 
\[\lim_{n\rightarrow\infty}{t_{A_n}(\epsilon)\over t_{A_n}(1-\epsilon)}=1\] for every $\epsilon\in(0,{1\over 2})$. 
\end{dfn}
See also Definition 3.3 in [5].

Our reasoning about the notion of sharp cutoff is inspired by the following two extreme scenarios.
Consider on one hand a family of Markov matrices $\{A_n\}$ with $n$ eigenvalues $1, 1- \eps, 0, 0, \dots 0$, 
and on the other hand a family $\{B_n\}$ with $n$ eigenvalues $1, 1-\eps, 1-\eps, \dots , 1-\eps$.  For fixed $n$, 
beginning with a vector ${\bf x_0}=(1,0,\ldots,0)$ we are interested in how quickly the vectors $A_n^i{\bf x_0}$ and 
$B_n^i{\bf x_0}$ approach the stationary distribution $(\frac{1}{n}, \frac{1}{n}, \dots, \frac{1}{n})$ when measured 
in $\ell^1$ norm.  For the moment, let us imagine that we can take $\sqrt{n}$ times the $\ell ^2$ norm as a proxy for 
the $\ell ^1$ norm; in general this substitution is not rigorous but it does hold quite tightly in many cases 
(see Chapter 3 of \cite{d2}), and doing it here allows us to change basis and analyze our scenarios in the 
diagonal basis of the Markov matrix.
Denote the image of ${\bf x_0}$ in the diagonal basis by ${\bf w}$; the image of the stationary distribution 
$(\frac{1}{n}, \frac{1}{n}, \dots \frac{1}{n})$ is the first eigenvector $(\frac{1}{\sqrt{n}}, 0, 0, \dots , 0)$.  

Then we have for $A_n$,  
\[a^i\coloneqq A_n^i({\bf w})= (w_1, (1-\eps)^iw_2, 0, \dots 0)\] 
for $i\geq 1$.  Under the assumption that the $|w_i|$ are each on the order of  $\frac{1}{\sqrt{n}}$, the $\ell^2$ 
distance to the stationary distribution is about $(1-\eps)^i w_2$, and our proxy measure is already within a 
constant of $(\frac{1}{\sqrt{n}}, 0, \dots, 0)$ at time $i=1$. Further iteration moves closer to stationary at basically 
the constant multiplicative rate $(1-\eps)$ per time step.  This does not display sharp cutoff, since $A_n$ moves 
quickly to within a constant distance of the stationary.  (See Definition \ref{def}.)

In contrast, for $B_n$, we have that 
\[b^i \coloneqq B_n^i({\bf w})= (w_1, (1-\eps)^iw_2, (1-\eps)^iw_3, \dots (1-\eps)^iw_n) \] 
has $\ell ^2$ distance $(1-\eps)^i \sqrt{\sum_{i=2}^n |w_i|^2 }\approx (1-\eps)^i$ which only gets within a constant
of stationary (in our proxy measure) at a time $i$ for which $(1-\eps)^i \sqrt{n}$ is constant, i.e., for 
$i=O(\log n)/\log \frac{1}{1-\eps}$.  Once $i$ exceeds this time, $b^i$ moves towards the stationary distribution at the 
same rate per step as $a^i$.   Thus $B_n$ spends a long time ($O(\log n)$) getting close to the stationary
relative to the time spent improving that closeness, meaning that $B_n$ does exhibit sharp cutoff.  

These examples illustrate the perspective that sharp cutoff is a criterion that captures those scenarios where the set of 
eigenvalues of the process drop off ``slowly enough.''  This perspective is made rigorous in our context 
via Lemma \ref{l:3}, which uses a standard argument to relate sharp cutoff to the decay of eigenvalues.  As we will 
then show, the eigenvalues for the type $r$ processes we consider here drop off too quickly for sharp cutoff to occur.

Our argument rests on one key idea which we now describe.
The fact that our groups $G$ are Abelian allows us to describe the eigenvalues of the process explicitly via the 
one dimensional representations of $G$.   Rather than analyzing directly the distribution of the sizes of these 
eigenvalues, we observe that these eigenvalues correspond in a nice way to certain vectors in an 
$r$ dimensional lattice.  We show that the sizes of the eigenvalues fall off more quickly than the set 
$e^{-c|x|^2}$ for $x$ ranging over the lattice, which despite being an infinite set is more readily summable.  
The lengths of lattice vectors shrink quickly enough to rule out sharp cutoff via this comparison.

Most of the work to prove Theorem 1 already shows up in the cyclic case, which we prove in Section \ref{s:cyclic}.
The general Abelian case, which is notationally more complicated, is completed in Section \ref{s:abelian}.

\section{Relating rate of convergence to eigenvalues}

We start with a general lemma relating the distance to the stationary distribution at time $t$ to the eigenvalues of 
the Markov matrix $A$ corresponding to a random walk on a Cayley graph.

We write
\[\{\lambda_k\}\subseteq(-1,1] \mbox{ for the eigenvalues of $A$ with } \lambda_0=1 \mbox{ and }
|\lambda_m|=\max_{k\not=0}|\lambda_k|<1.\]  

\begin{lem} \label{l:3} Given a transition matrix $A$ for a random walk on a Cayley graph for 
the Abelian group $G$, we have
\begin{equation} \label{e:1}
\lambda_m^{2t}\leq d_A^2(t)\leq \sum_{k\not=0}\lambda_k^{2t}.
\end{equation}
\end{lem}  

\begin{proof}
For the left inequality note that $A=A^*$ is self adjoint and the stationary distribution is ${\bf v}_0={1\over n}{\bf 1}$. 
Write ${\bf v}_m$ for the eigenvector with $A{\bf v}_m=\lambda_m{\bf v}_m$ and $|{\bf v}_m|_1=1$.  Since $A$ is a 
transition matrix for a random walk on a Cayley graph for $G$ it commutes with rotation (action by $G$).  Thus 
${\bf v}_m$ is an eigenvector for rotation and hence all entries of ${\bf v}_m$ have the same norm, which by the 
normalization is $|{\bf v}_m|_\infty={1\over n}$.  Since 
$d_A(t)=\max_{\bf v}{\langle A^t{\bf x}_0-{\bf v}_0,{\bf  v}\rangle\over |{\bf v}|_\infty}$ this gives 
$d_A(t)\geq n|\langle A^t{\bf x}_0-{\bf v}_0,{\bf  v}_m\rangle|=|\lambda_m|^t$.  

For the right inequality if ${\bf v}_k$ is the eigenvector of $A$ with eigenvalue $\lambda_k$ and every entry having 
norm ${1\over n}$ then $|\langle{\bf x}_0,{\bf v}_k\rangle|={1\over n}$ so that 
$d_A^2(t)=|A^t{\bf x}_0-{\bf v}_0|_1^2\leq n|A^t{\bf x}_0-{\bf v}_0|_2^2=\sum_{k\not=0}\lambda_k^{2t}$.
\end{proof}

\section{The cyclic case} \label{s:cyclic}
With Lemma \ref{l:3} in hand, we first prove Theorem 1 in the case where every
group is cyclic.  In this case the Markov matrices are circulants. Focusing on one walk from
the family, we'll denote by  $\{\pm a_i\}\ (1\le i \le r)$ the $r$ possible steps of the walk  on $\Z\slash n\Z$ 
and by $A$ the corresponding symmetric Markov transition matrix.  The Fourier transform yields  an explicit 
form of the eigenvalues of $A$ for $k=1, \dots, n$:
\[\lambda_k=\sum_{i=1}^r{2\over 2r+1}\cos\left(2\pi{ka_i\over n}\right)+{1\over 2r+1}.\]  

Our strategy is to bound the $\lambda_k$ and use these bounds in conjunction with Lemma \ref{l:3} to get 
upper and lower bounds on $d_A^2(t)$ which will be tight enough to show that sharp cutoff 
does not occur.  Specifically, we will associate to each $\lambda_k$ an 
$r$ dimensional vector $\sigma_k$ for which $e^{-c_1|\sigma_k|^2} \leq \lambda_k \leq e^{-c_2|\sigma_k|^2}$.  
Modulo a minor complication, the
$\sigma_k$ will all lie in an $r$ dimensional lattice, with $\sigma_m$ being a minimal length vector in that lattice.  
This will allow us to upper bound the right hand side of \eqref{e:1} by the sum over the whole lattice of 
$e^{-c_2|\sigma_k|^2}$.  Despite the inclusion of many extra terms, this bound will be tight enough for our needs.

The minor complication is that if $\lambda_k<0$ we need to analyze a slightly different lattice which we do via
a slightly different vector $\tilde{\sigma}_k$. However the core of the argument remains the same.

We begin by defining for each $k=1,\dots, n-1$ the vector
$$\sigma_k = \bigl(\langle \frac{ka_1}{n} \rangle, \dots, \langle \frac{ka_r}{n} \rangle\bigr) \in \R^r$$
where we use $\langle x\rangle\in(-{1\over 2},{1\over 2}]$ for the 
smallest translate of $x$ by an integer.  The vector $\sigma_k$ lives in the lattice
$$\Lambda= \Z^r + \Z\cdot(\frac{a_1}{n} , \dots, \frac{a_r}{n} \bigr) .$$
Note that if all the coordinates of $\sigma_k$ are small, then all the cosines in the expression for $\lambda_k$ are 
close to $1$, so $\lambda_k$ is close to 1.  For $\lambda_k$ to be close to $-1$, we are forced to have the coordinates 
of $\sigma_k$ close to $\pm \frac{1}{2}$.   For this reason, we also consider the vector
$$\tilde{\sigma}_k = \bigl(\langle \frac{ka_1}{n} +\frac12 \rangle, \dots, \langle \frac{ka_r}{n}+\frac12 \rangle\bigr) \in \R^r,$$
which lives in the lattice
$$\tilde{\Lambda}= \Z^r + \Z\cdot\bigl(\frac{a_1}{n} , \dots, \frac{a_r}{n} \bigr) + \Z\cdot\bigl(\frac12, \dots, \frac12\bigr) .$$
Let $\mu$ and $\tilde{\mu}$ denote the lengths of the shortest nonzero vectors of $\Lambda$ and $\tilde{\Lambda}$
respectively.  Since $\sigma\in\tilde\Lambda$ implies $2\sigma\in\Lambda$, we have
$$2\tilde\mu\geq \mu.$$

The following lemma gives a lower bound on $d(t)$ by bounding $|\lambda_m|$ in terms of the length 
$\mu$ of the shortest vector in $\Lambda$.  Here we assume $n$ is sufficiently large, which we may do since 
if any subfamily of walks has bounded size then the family cannot have sharp cutoff.

\begin{lem}  Let $\beta=\frac{8\pi^2}{(2r+1)}$.
Then if $n$ is sufficiently large, 
 $$|\lambda_m| \geq e^{-\beta\mu^2}.$$
\end{lem}  

\begin{proof}
First note that $\hbox{Vol}[(\R^r\slash \Lambda)]={1\over n}$, which approaches $0$.  It follows that for large 
enough $n$, the shortest nonzero vector in $\Lambda$ has length less than $1$, and is therefore equal to 
$\sigma_k$ for some $k=1,\dots,n-1$. For this shortest $\sigma_k$, we may also assume that the all the 
coordinates of $\sigma_k$ have absolute value less than $\frac1{2\pi}$.  We will now show 
$$|\lambda_k| \geq e^{-\beta|\sigma_k|^2},$$
which will establish the lemma.

For any $x\in[-1,1]$, we have $\cos(x)\geq e^{-x^2}.$  Thus
$$\lambda_k \geq \frac1{2r+1}\biggl( \sum_i 2e^{-(2\pi\langle\frac{ka_i}{n}\rangle)^2}  + 1 \biggr)$$
Now consider the right hand side as an average of $2r+1$ values of the function $e^{-x}$.  The convexity of 
this function yields
$$\lambda_k \geq e^{-\frac{2}{2r+1}\sum_i(2\pi\langle\frac{ka_i}{n}\rangle)^2} = e^{-\beta|\sigma_k|^2},$$
as desired.
\end{proof}

Next we establish an upper bound on $d(t)$ by bounding the sum of $\lambda_k^{2t}$.  For each $k$, 
we bound $|\lambda_k|$ by an exponential of either $|\sigma_k|$ or $|\tilde\sigma_k|$. We then replace the sum 
of these exponentials with a sum over the entire lattice $\tilde\Lambda$.

\begin{lem}  
Let $\alpha=\frac{8}{(2r+1)^2}$.  Then
 \begin{equation} \label{e:4} 
 \sum_{k\not=0}\lambda_k^{2t}\leq 2\sum_{\sigma\in\tilde\Lambda\setminus\{0\}} e^{-\alpha|\sigma|^2\cdot 2t}.
 \end{equation}
\end{lem}

\begin{proof}
Fix $k\in\{1,\dots,n-1\}$.  First assume $\lambda_k\geq 0$.
As in the previous proof, we will replace cosines with exponentials, this time using the inequality 
$\cos(x)\leq e^{-{1\over 2\pi^2}x^2}$, which is valid for all $x\in[-{3\over 2}\pi,{3\over 2}\pi]$.  This gives
$$\lambda_k \leq \frac1{2r+1}\biggl( \sum_i 2e^{-2\langle\frac{ka_i}{n}\rangle^2}  + 1 \biggr)$$
We now consider the right hand side as an average value of the function $e^{-x^2}$ at $2r+1$ values of $x$, each 
of which lies in the interval $[0,{1\over\sqrt{2}}]$.  Since  $e^{-x^2}$ is concave down on this interval, we obtain
$$\lambda_k \leq e^{-\frac{8}{(2r+1)^2} [\sum |\langle\frac{ka_i}{n}\rangle  |  ]^2} = e^{-\alpha |\sigma_k |_1^2},$$
and where $|v|_1$ denotes the 1-norm of $v$.  Since $|v|_1\geq |v|$, we finally arrive at
$$\lambda_k \leq e^{-\alpha |\sigma_k |^2}.$$

Now consider the case in which $\lambda_k<0$.  In this case, $|\lambda_k|=-\lambda_k$ is the sum of 
negative cosines, so using $-\cos(2\pi x) = \cos(2\pi(x+\frac12))$, an argument similar to the first case shows that
$$|\lambda_k| \leq e^{-\alpha |\tilde\sigma_k |^2}.$$

Hence in both case we see that $|\lambda_k|$ is bounded above by $e^{-\alpha |\sigma |^2}$, with 
$\sigma\in\tilde\Lambda$.  To complete the proof of the lemma, we must show that no such $\sigma$ 
appears for more than two different values of $k$. To see this, note if $\sigma_k=\sigma_{k'}$ for some 
$k,k'\in\{1,\dots,n-1\}$, then $ka_i \equiv k'a_i \mod n$ for all $i$.  Since the $a_i$ generate $\Z_n$, 
this forces $k=k'$.  Similarly, $\tilde\sigma_k=\tilde\sigma_{k'}$ implies $k=k'$.  Finally if $\sigma_k=\tilde\sigma_{k'}$, 
then we get either $k=k'$ or $k=k'\pm \frac{n}{2}$, with $n$ even.  In any event, no more than two distinct choices of 
$k$ lead to the same lattice element $\sigma$. 
\end{proof}

%
%
%
%

The next two lemmas will further bound the right hand side of \eqref{e:4} in terms of $r$ and $\Lambda$.
This amounts to showing that the number of lattice points of length at most $d$ is bounded above by a polynomial 
in $d$.  Given two discrete (infinite) multisets $S, T \subset \R_{\ge 0}$, we say $S$ \emph{dominates} $T$ if there 
are orderings of $S= \{s_i: i\in \Z ^+\}$ and $T=\{t_i: i \in \Z ^+\}$ such that $s_i\geq t_i$ for all $i$.  Note that 
dominance induces a partial order on such multisets.  

For a lattice $\Lambda$, let $|\Lambda|$ denote the multiset of norms of vectors of $\Lambda$.  So $|\Lambda|$
is discrete and contained in $\R_{\ge 0}$.

%
%
%

\begin{lem}  \label{l:dominate}Let $r$ be a fixed integer.  Let $T$ be the multiset consisting of $\{0 \}$ 
and $3^r(i+1)^r$ copies of $i\in \N= \{1,2, \dots \}$.  Then for every full rank lattice 
$\Lambda$ in $\R^r$ with minimal nonzero vector of length $1$, the set $|\Lambda|$ dominates the 
multiset $T$. \end{lem}
\begin{proof}
Observe that for $d\ge 1$ the number $n(d)$ of lattice points of $\Lambda$ of norm no greater than 
$d$ is bounded above by $(3d)^r$.  This is because the shortest vector of $\Lambda$ 
has length 1, so radius $1/2$ balls around points of $\Lambda$ are disjoint, so
$n(d) \cdot B(1/2) \le B(d+1/2)$ where $B(t)$ is the volume of a ball in $\R^r$ of radius $t$.
This gives $n(d)\le B(d+1/2) / B(1/2) = 2^r(d+1/2)^r \le 3^r d^r$.

It follows that  no more than $3^r (i+1)^r $ lattice points are between lengths $i$ and $i+1$ and 
thus $|\Lambda |$ dominates the multiset $T$. 
\end{proof}

\begin{lem}
Fix $r$ and $\Lambda$ a full rank lattice in $\R^r$, and let $\mu=\min_{{\bf 0}\not=\sigma\in\Lambda}\{|\sigma|_2\}$. 
Then \[\sum_{\sigma\in\Lambda}e^{-a|\sigma|_2^2}\leq 1+3^rr! (\frac{1}{(1-(e^{-a\mu^2}))^r} -1)\]
\end{lem}

\begin{proof}
By  Lemma  \ref{l:dominate} we have:
\[ \sum_{\sigma\in\Lambda}e^{-a|\sigma|_2^2}\leq 1 + \sum_{t\in T} e^{-a(\mu t)^2} \leq 1 + \sum_{i\in \N}3^r (i+1)^r (e^{-a\mu^2})^{i^2}. \]  
This last sum can be bounded as
\[\sum_{i\in \N}3^r (i+1)^r (e^{-a\mu^2})^{i^2} \leq \sum_{j\in \N} 3^r(\sqrt{j} +1)^r(e^{-a\mu^2})^{j}  \leq 3^r r!\sum_{j\in \N}  \binom{j+r}{r}  (e^{-a\mu^2})^{j} \]
where we've substituted $i^2=j$ for the first inequality and $(\sqrt{j} +1)^r \leq r! \binom{j+r}{r}$ for the second.  
Finally using the well known $\sum_{i=1}^{\infty} \binom{i+r}{i} x^i=\frac{1}{(1-x)^r} -1$, we have
\[ 3^r r!\sum_{j\in \N}  \binom{j+r}{r}  (e^{-a\mu^2})^{j} =3^rr! (\frac{1}{(1-(e^{-a\mu^2}))^r} -1).\]
\end{proof}

%
We now prove Theorem 1 (in the cyclic case).
Combining Lemma 5 with Lemma 7 yields the upper bound
\begin{equation*}
\sum_i \lambda_i ^{2t} \leq 2\cdot 3^rr! (\frac{1}{(1-e^{-\alpha\tilde\mu^2\cdot 2t})^r} -1) 
\end{equation*}  
Meanwhile, Lemma 4 together with $2\tilde\mu\geq\mu$ gives
$$e^{-\alpha\tilde\mu^2\cdot 2t} \leq |\lambda_m|^{\frac{\alpha}{2\beta}t},$$
hence
\begin{equation} \label{e:2}
\sum_i \lambda_i ^{2t} \leq  2\cdot 3^rr! (\frac{1}{(1-|\lambda_m|^{\frac{\alpha}{2\beta}t})^r} -1)
\end{equation}

By the first inequality of Lemma \ref{l:3} we have $t_A(\eps) \geq t_0$ where $t_0$ is the solution to 
$|\lambda_m^t| = \eps$.  It follows that 
\[ t_A(\eps) \geq \log _{|\lambda_m|} (\eps).\]

By the second inequality of Lemma 3 and equation \eqref{e:2}, we have $t_A(1-\eps) \leq t_1$ with $t_1$ 
the solution to $  2\cdot 3^rr! (\frac{1}{(1-|\lambda_m|^{\frac{\alpha}{2\beta}t})^r} -1) =(1-\eps)^2$.  Solving this yields
\begin{equation} \label{e:tup}  
t_A(1-\eps) \leq \frac{2\beta}{\alpha} \log_{|\lambda_m|} C(\eps), 
\end{equation}
with $C(\eps)= ( 1- (\frac{(1-\eps)^2}{ 2\cdot 3^rr!} +1)^{-1/r}) $.

Together we have for each walk in the family,
\[ \frac{ t_A(\eps)}{ t_A(1-\eps)} \geq \frac{\alpha}{2\beta} \log _{|\lambda_m|} (\eps)/ \log_{|\lambda_m|}C(\eps) = \frac{1}{2\pi^2(2r+1)} \log _{C(\eps)} \eps. \]
Notice that as $\eps \rightarrow 0$ the $C(\eps) \rightarrow ( 1- (\frac{1}{ 2\cdot 3^rr!} +1)^{-1/r}) $ which is less than $1$ 
and bounded away from zero, so there will be a choice of $\eps$ that will make the right hand side greater than $1$ 
independent of $A$.  This completes the proof in the cyclic case.

\section{The general case} \label{s:abelian}

To prove the main theorem for arbitrary Abelian groups, we closely follow the proof for cyclic groups given above.

We consider an arbitrary finite Abelian group $G$, which we express as the product of $s$ cyclic groups of orders 
$n_1, \dots, n_s$.  Let $n=|G|=n_1\cdots n_s$.  We suppose that we have $r$ generators $a_1, \dots, a_r$ of $G$, with 
$$a_j = (a_{j1}, \dots a_{js}),$$
with each $a_{jh}\in\Z_{n_h}$, $h=1,\dots,s$.   We will assume that we have chosen a product decomposition in 
which $s\leq r$, which is always possible for an Abelian group generated by $r$ elements.

It will be convenient to have 
$$\eta_{jh} = \frac{a_{jh}}{n_h} \in [0,1),$$
and define $\eta_j=(\eta_{j1}, \dots \eta_{js})$.

The eigenvalues of our Markov process are indexed by tuples $k=(k_1, \dots, k_s)$ with $k_h\in\Z_{n_h}$ and are given by
$$\lambda_k = \frac1{2r+1} \biggl(  2\sum_{j=1}^r \cos (2\pi k \cdot \eta_j) + 1  \biggr).$$

As before we identify the $r$-tuple appearing in the arguments of the cosines above, letting
$$\sigma_k = \bigl(\langle k\cdot \eta_1 \rangle, \dots, \langle k\cdot \eta_r \rangle\bigr) \in \R^r.$$
The vector $\sigma_k$ lives in the lattice
$$\Lambda= \Z^r + \sum_{h=1}^s \Z \theta_h,$$
where
$$\theta_h =(\eta_{1h}, \dots, \eta_{rh}).$$
(The $\theta$ matrix is the transpose of the $\eta$ matrix.)

As before, we also introduce
$$\tilde{\sigma}_k = \bigl(\langle k\cdot \eta_1 +\frac12\rangle, \dots, \langle k\cdot \eta_r +\frac12 \rangle\bigr) \in \R^r,$$
which lives in the lattice
$$\tilde{\Lambda}= \Z^r  + \sum_{h=1}^s \Z \theta_h +  \Z\cdot\bigl(\frac12, \dots, \frac12\bigr).$$

As before, we let $\lambda_m$ be the eigenvalue with largest absolute value, $m\neq 0$, and we have
$$\lambda_m^{2t} \leq d^2(t) \leq \sum_{k\neq 0} \lambda_k^{2t}.$$

With this setup, the rest of the argument is similar to the cyclic case.  The conclusion of Lemma 4 holds exactly 
as before, while Lemma 5 must be modified as follows:

\begin{lem}  Let $\alpha=\frac{8}{(2r+1)^2}$.
Then
 $$\sum_{k\not=0}\lambda_k^{2t}\leq 2^s\sum_{\sigma\in\tilde\Lambda\setminus\{0\}} e^{-\alpha|\sigma|^2\cdot 2t}.$$
\end{lem}  
\begin{proof}
As in the  proof of Lemma 5, we see that if $\lambda_k>0$, then
$$\lambda_k \leq e^{-\alpha |\sigma_k |^2}.$$
and if $\lambda_k<0$, then 
$$|\lambda_k| \leq e^{-\alpha |\tilde\sigma_k |^2}.$$

So in either case $|\lambda_k|$ is bounded above by $e^{-\alpha |\sigma |^2}$, with $\sigma\in\tilde\Lambda$.  
To complete the proof of the lemma, we must show that no such $\sigma$ appears for more than $2^s$ different 
values of $k$. For this purpose, assume that $\sigma_k=\sigma_{k'}$.  This implies that for $j=1,\dots,r$,
$$\langle k\cdot \eta_j \rangle = \langle k'\cdot \eta_j \rangle$$
Equivalently, if we consider the $r\times s$ matrix $A=(a_{jh})$, and let $x$ be the length $s$ column vector 
with $x_h=\frac{k_h-k'_h}{n_h}$, then 
$$Ax\in \Z^r.$$
Now we use the fact the $a_j$ generate $G$.   This implies that there exists an integer vector $c=(c_1,\dots,c_r)$ 
such that $\sum c_j a_j = (1,0, \dots, 0)$ in $G$, i.e.,
$$cA = (1+\gamma_1 n_1, \gamma_2 n_2, \dots, \gamma_s n_s),$$
where the $\gamma_h$ are integers.  
It follows that 
$$cAx = x_1+\sum_{h=1}^s \gamma_h (k_h-k'_h),$$
and hence $x_1\in\Z$.  Similarly, all $x_h\in Z$, and we get $k_h \equiv k'_h \mod n_h$, i.e. $k=k'$.
 
We have now established that $\sigma_k=\sigma_{k'}$ implies $k=k'$.  A similar argument shows that 
$\tilde\sigma_k=\tilde\sigma_{k'}$ implies $k=k'$.  Finally if $\sigma_k=\tilde\sigma_{k'}$, then for each $h$, we get 
either $k_h=k'_h$ or $k_h=k'_h\pm \frac{n_h}{2}$, with $n_h$ even. It follows that no more than $2^s$ distinct 
choices of $k$ lead to the same lattice element $\sigma$. 
\end{proof}

Finally, because $r$ is constant and $s\leq r$, the presence of the factor $2^s$ in the above lemma does not effect 
the rest of the proof given in the cyclic case, which goes through without further modification.

%
%
%
%

\section{Remarks}
Note that the same bound as \eqref{e:tup} gives $t_A({1\over 2})\leq {c\over -\ln|\lambda_m|}$ for some fixed constant $c$ so that $(1-|\lambda_m|)t_A({1\over 2})$ is uniformly bounded so this class of examples agrees with Peres' conjecture in [3].  

The authors thank Persi Diaconis for suggesting this problem and Julie Landau for her court coverage.

\bibliographystyle{amsplain}

\end{document}